\documentclass[12pt,letter]{amsart}

\usepackage{amssymb}
\usepackage{graphicx}
\usepackage{accents}
\usepackage[english,german]{babel}
\usepackage{color}
\usepackage{cancel}
\usepackage{url}

\newtheorem{theorem}{Theorem}[section]
\newtheorem{lemma}[theorem]{Lemma}
\newtheorem{Prop}[theorem]{Proposition}

\theoremstyle{definition}
\newtheorem{definition}[theorem]{Definition}

\theoremstyle{remark}
\newtheorem{remark}[theorem]{Remark}

\numberwithin{equation}{section}

\makeatletter
\newcommand*{\rom}[1]{\text{\expandafter\@slowromancap\romannumeral #1@}}
\makeatother
\newcommand{\R}{\mathbb{R}}

\newcommand{\Z}{\mathbb{Z}}

\newcommand{\Ric}{\operatorname{Ric}}

\newcommand{\Scal}{\operatorname{R}}
\newcommand{\tr}{\operatorname{tr}}

\newcommand{\mylap}[1]{{}^{#1}\!\triangle}
\newcommand{\my}[2]{{}^{#1}{#2}}

\newcommand{\free}[1]{\accentset{\,\circ}{#1}}

\newcommand{\spacetime}{(\mathfrak{L}^{4},\mathfrak{g})}

\newcommand{\photo}{P^{3}}
\newcommand{\Sphoto}{\overline{P}^{3}}

\newcommand{\slice}{M^{3}}

\newcommand{\surf}{\Sigma^{2}}
\newcommand{\glue}{\Sigma^{2}_{i}}
\newcommand{\schild}{Schwarz\-schild }

\newcommand{\spacet}{\mathfrak{L}^{4}}

\newcommand\beq{\begin{equation}}
\newcommand\eeq{\end{equation}}
\newcommand\ben{\begin{enumerate}}
\newcommand\een{\end{enumerate}}
\newcommand\bit{\begin{itemize}}
\newcommand\eit{\end{itemize}}

\begin{document}
\selectlanguage{english}
\title[Uniqueness of photon spheres]{Uniqueness of photon spheres via \\positive mass rigidity}
\author{Carla Cederbaum}
\address{Mathematics Department, Universit\"at T\"ubingen, Germany}
\email{cederbaum@math.uni-tuebingen.de}
\thanks{The first author is indebted to the Baden-W\"urttemberg Stiftung for the financial support of this research project by the Eliteprogramme for Postdocs.}

\author{Gregory J. Galloway}
\address{Mathematics Department, University of Miami, USA}
\email{galloway@math.miami.edu}
\thanks{The second author was partially supported by NSF grant DMS-1313724.}

\date{}

\begin{abstract}
In a recent paper the first author established the uniqueness of photon spheres, suitably defined, in static vacuum asymptotically flat spacetimes by adapting Israel's proof of static black hole uniqueness. In this note we establish uniqueness of photon spheres by adapting the argument of Bunting and Masood-ul-Alam~\cite{BMuA}, which then allows certain assumptions to be relaxed.   
In particular, multiple photon spheres are allowed a priori. 

As a consequence of our result, we can rule out the existence of static configurations involving multiple ``very compact'' bodies and black holes.\end{abstract}

\maketitle

\section{Introduction}\label{sec:intro}
The static spherically symmetric \schild black hole spacetime\footnote{The same formula still defines a \emph{\schild spacetime} if $m<0$. The corresponding metric is well-defined on $\overline{\mathfrak{L}}^{4}=\R\times(\R^{3}\setminus\lbrace0\rbrace)$ but possesses neither a black hole horizon nor a photon sphere. If $m=0$, the \schild spacetime degenerates to the Minkowski spacetime.} of mass $m>0$ can be represented as
\begin{align}
(\overline{\mathfrak{L}}^{4}:=\R\times(\R^{3}\setminus B_{2m}(0)),\overline{\mathfrak{g}}),
\end{align}
where the Lorentzian metric $\overline{\mathfrak{g}}$ is given by
\begin{align}\label{schwarzmetric}
\overline{\mathfrak{g}}&=-\overline{N}^{2}dt^{2}+\overline{N}^{-2}dr^{2}+r^{2}\Omega,\quad 
\overline{N}=\sqrt{1-\frac{2m}{r}},
\end{align}
with $\Omega$ denoting the canonical metric on $\mathbb{S}^{2}$. 
The black hole event horizon occurs at  $r = 2m$, where the static coordinates degenerate. The timelike submanifold $\Sphoto:=\R\times\mathbb{S}^{2}_{3m}=\lbrace r=3m\rbrace$ is called a \emph{photon sphere} because any null geodesic of $(\overline{\mathfrak{L}}^{4},\overline{\mathfrak{g}})$ that is initially tangent to $\Sphoto$ remains tangent to it. The \schild photon sphere thus models (an embedded submanifold ruled by) photons spiraling around the central black hole ``at a fixed distance''. 

The \schild photon sphere and the notion of trapped null geodesics in general are crucially relevant for questions of dynamical stability in the context of the Einstein equations. Moreover, photon spheres are related to the existence of relativistic images in the context of gravitational lensing. Please see \cite{CederPhoto,yazadjiev} and the references cited therein for more information on photon spheres.

To the best knowledge of the authors, it is mostly unknown whether more general spacetimes can possess (generalized) photon spheres, see p.\;838 of \cite{CVE}. Recently, the first author gave a geometric definition of photon spheres in static spacetimes and proved uniqueness of photon spheres in $3$-dimensional asymptotically flat static vacuum spacetimes under the assumption that the lapse function $N$ of the spacetime regularly foliates the region exterior to the photon sphere (Theorem 3.3 in \cite{CederPhoto}). The proof in \cite{CederPhoto} extends Israel's proof of black hole uniqueness \cite{Israel} to the context of photon spheres. In particular, as in Israel's proof,  the condition that the lapse function $N$ regularly foliates the region exterior to the black hole was required, and implies a priori that there is only one photon sphere in the spacetime. 

Adopting the definition of photon spheres in \cite{CederPhoto}, we will prove photon sphere uniqueness for $3+1$-dimensional  asymptotically flat static vacuum, or \emph{geometrostatic}, spacetimes $\spacetime$ without assuming that the lapse function $N$ regularly foliates the spacetime. In particular, we allow a priori the possibility of multiple photon spheres. 
To accomplish this we make use of the rigidity case of the Riemannian positive mass theorem (under the weaker regularity assumed in \cite{Bartnik}, see also \cite{LeeLefloch} and references cited therein), in a manner similar to the proof of black hole uniqueness in the static case due to Bunting and Masood-ul-Alam~\cite{BMuA}. 

This paper is organized as follows. In Section \ref{sec:setup and definition}, we will recall the definition and a few properties of photon spheres in geometrostatic spacetimes from \cite{CederPhoto}. In Section~\ref{sec:proof}, we will prove that the only geometrostatic spacetime admitting a photon sphere is the \schild spacetime:
{\renewcommand{\thetheorem}{\ref{thm:main}}
\begin{theorem}
Let $\spacetime$ be a geometrostatic spacetime that possesses a (possibly disconnected) photon sphere $(\photo,p)\hookrightarrow\spacetime$, arising as the inner boundary of~$\spacet$. Let $m$ denote the ADM-mass of $\spacetime$ and let $\mathfrak{H}:\photo\to\R$ denote the mean curvature of $(\photo,p)\hookrightarrow\spacetime$. Then $m=(\sqrt3\,\mathfrak{H})^{-1}$, with $\mathfrak{H}>0$, and $\spacetime$ is isometric to the region $\lbrace r\geq3m\rbrace$ exterior to the photon sphere $\lbrace r=3m\rbrace$ in the \schild spacetime of mass~$m$. In particular, $(\photo,p)$ is connected and a cylinder over a topological sphere.
\end{theorem}
\addtocounter{theorem}{-1}}

\begin{remark} 
As in \cite{CederPhoto}, one does not need to assume a priori that the mass is positive;  this is a consquence of the theorem.  In particular, the existence of photon spheres in static spacetimes of non-positive mass is ruled out.
\end{remark}

\begin{remark}\label{rem:multiple}
In addition to the photon sphere inner boundary, our arguments would allow for the presence of a (Killing)  horizon  as additional components of the boundary of the spacetime.  These would be treated just as in the original argument by Bunting and Masood-ul-Alam \cite{BMuA}. For simplicity, we have assumed no horizon boundaries.
\end{remark}

In Section \ref{sec:nbody}, we explain how Theorem \ref{thm:main} can be applied to the so called static $n$-body problem in General Relativity, yielding the following corollary:

{\renewcommand{\thetheorem}{\ref{thm:nbody}}
\begin{theorem}[No static configuration of $n$ ``very compact'' bodies and black holes]
There are no static equilibrium configurations of $n>1$ bodies and black holes in which each body is surrounded by its own photon sphere.
\end{theorem}
\addtocounter{theorem}{-1}}

\section{Setup and definitions}\label{sec:setup and definition}
Let us first quickly review the definition of and some facts about asymptotically flat static vacuum spacetimes. These model exterior regions of static configurations of stars or black holes. See Bartnik \cite{Bartnik} for a more detailed account of asymptotically flat Riemannian manifolds and harmonic coordinates as well as for the definition of the weighted Sobolev spaces $W^{k,p}_{-\tau}(E)$ we will use in the following. More details and facts on asymptotically flat static vacuum spacetimes can be found in \cite{CDiss}.

\begin{definition}[Geometrostatic spacetimes and systems]\label{def:AFSVS}
A smooth, time-oriented Lorent\-zian manifold or \emph{spacetime} $\spacetime$ 
is called \emph{(standard) static} if there exists a smooth Riemannian manifold $(\slice,g)$ 
and a smooth \emph{lapse} function $N:\slice\to\R^{+}$ such that
\begin{align}\label{static}
\spacet&=\R\times \slice,\quad \mathfrak{g}=-N^{2}dt^{2}+g.
\end{align}
It is called \emph{vacuum} if it satisfies the Einstein vacuum equation
\begin{align}\label{EE}
\mathfrak{Ric}&=0,
\end{align}
where $\mathfrak{Ric}$ denotes the Ricci curvature tensor of $\spacetime$. We will sometimes call $\slice$ a \emph{(time-)slice} of $\spacet$, as it arises as $\slice=\lbrace t=0\rbrace$, where $t$ is considered as the time variable of the spacetime. It will be convenient for us to allow $M^3$ to be a manifold with boundary.  In this case it is to be understood that $\spacetime$ extends to a slightly larger standard static spacetime containing $\R \times \partial M$.

A static spacetime is called \emph{asymptotically flat} if the manifold $\slice$ is diffeomorphic to the union of a compact set and an open \emph{end}
  $E^3$ which is diffeomorphic to $\R^3\setminus \overline{B}$, where $B$ is the open unit ball in $\R^3$. Furthermore, we require that, in the end $E^3$, the lapse function $N$, the metric $g$, and the coordinate diffeomorphism $\Phi=(x^{i}):E^3\to\R^3\setminus \overline{B}$ combine such that
{\begin{align}\label{AF}
g_{ij}-\delta_{ij}&\in W^{k,q}_{-\tau}(E)\\\label{NAF}
N-1\;\;&\in W^{k+1,q}_{-\tau}(E)
\end{align}
for some $\tau>1/2$, $\tau\notin\Z$, $k\geq2$, $q>4$,} and that $\Phi_{\ast}g$ is uniformly positive definite and uniformly continuous on $\R^3\setminus \overline{B}$. Here, $\delta$ denotes the Euclidean metric on $\R^3$. For brevity, smooth\footnote{M\"uller zum Hagen \cite{MzH} showed that static spacetimes with $g_{ij}, N\in C^3$ are automatically real analytic with respect to wave-harmonic coordinates if they solve \eqref{EE}.} asymptotically flat  static vacuum spacetimes will also be referred to as \emph{geometrostatic spacetimes}, the associated triples $(\slice,g,N)$ will be called \emph{geometrostatic systems}. We will use the radial coordinate $r:=\sqrt{(x^1)^{2}+(x^2)^{2}+(x^3)^{2}}$ corresponding to the coordinates $(x^{i})$.
\end{definition}

Exploiting \eqref{static}, the Einstein vacuum equation \eqref{EE} reduces to\begin{align}\label{SMEvac1}
N\,{\Ric}&={\nabla}^2 N\\\label{SMEvac3}
\triangle N&=0
\end{align}
on $\slice$, where $\nabla^{2}$, $\Ric$, and $\Scal$ denote the covariant Hessian, Ricci, and scalar curvature of the metric $g$, respectively. Combining \eqref{SMEvac1} and \eqref{SMEvac2}, one finds
\begin{align}\label{SMEvac2}
\Scal&=0
\end{align}
on $\slice$, where $\triangle$ denotes the Laplacian with respect to $g$. The \emph{static metric equations} \eqref{SMEvac1}, \eqref{SMEvac3} are a system of degenerate elliptic quasi-linear second order PDEs in the variables $N$ and $g_{ij}$ (with respect to for example $g$-harmonic coordinates). 

\subsection{Notation and conventions}
Our sign convention is such that the second fundamental form $h$ of an isometrically embedded $2$-surface $(\surf,\sigma)\hookrightarrow (\slice,g)$ with respect to the outward unit normal vector field $\nu$ is chosen such that
\begin{align}
h(X,Y)&:=g(\my{g}{\nabla}_{X} \nu,Y)
\end{align}
for all $X,Y\in\Gamma(\surf)$. We will make use of the contracted Gau{\ss} equation
\begin{align}\label{Gauss}
\my{g}{\Scal}-2\,\my{g}{\Ric}(\nu,\nu)&= \my{\sigma}{\Scal}-(\my{\sigma}{\tr}\,h)^2+\lvert {h}\rvert^2_{\sigma},
\end{align}
where the left upper indices indicate the metric from which a certain covariant derivative or curvature tensor is constructed. The trace-free part of $h$ will be denoted by $\free{h}$. We will also use the well-known identity
\begin{align}\label{surflap}
\mylap{g}f=\mylap{\sigma\!}f+\my{g}{\nabla}^2\!f(\nu,\nu)+H\nu(f)
\end{align}
which holds for any smooth function $f:\slice\to\R$.

\subsection{Definition of photon surfaces and photon spheres}
We adopt the definition of photon spheres from \cite{CederPhoto}. First, we recall the definition of photon surfaces by Claudel, Virbhadra and Ellis \cite{CVE}, see also Perlick \cite{Perlick}.

\begin{definition}[Photon surface]\label{def:photo-surf}
A timelike embedded hypersurface $(\photo,p)\hookrightarrow\spacetime$ of a smooth spacetime $\spacetime$ is called a \emph{photon surface} if any null geodesic initially tangent to $\photo$ remains tangent to $\photo$ as long as it exists.
\end{definition}

A photon sphere is then defined as a photon surface on which photons have constant energy and frequency. By Lemma 2.7 of \cite{CederPhoto}, this is equivalent to $N$ being constant along each connected component of a photon surface:

\begin{definition}[Photon sphere]\label{def:photo}
Let $\spacetime$ be a geometrostatic spacetime and $(\photo,p) \hookrightarrow\spacetime$ a photon surface. Then $\photo$ is called a \emph{(generalized) photon sphere} if the lapse function $N$ of the spacetime is constant on each connected component of~$\photo$.
\end{definition}

\begin{remark}
Note that the definition of \emph{(generalized) photon sphere} a priori neither requires the photon sphere to be connected nor to have the topology of a cylinder over a sphere. This allows us to treat multiple photon spheres at once. In fact, as a consequence of Proposition \ref{prop} and Lemma \ref{untrapped} below, we will see that 
 the topology of each component of a (generalized) photon sphere  is that of a cylinder over a sphere.
\end{remark}

\subsection{Properties of photon spheres}\label{subsec:props}
Let $(\photo,p)$ be a photon sphere arising as the inner boundary of a geometrostatic spacetime $\spacetime$ as in Theorem \ref{thm:main} and let $N_i:=N\vert_{\photo_{i}}$ denote the (constant) value of $N$ on the connected component $(\photo_i,p_i)$ of $(\photo,p)$ for all $i=1,\dots,I$. As the photon sphere $(\photo,p)$ arises as the inner boundary of a standard static spacetime according to Definition \ref{def:AFSVS}, each component $(\photo_{i},p_{i})$ must be a warped cylinder $(\photo_i,p_i)=(\R\times\surf_i,-N_i^2dt^2+\sigma_i)$,
where $\surf_i$ is the (necessarily compact) intersection of the photon sphere component $\photo_i$ and the time slice $\slice$ and $\sigma_i$ is the (time-independent) induced metric on $\surf_i$. Set $\surf:=\cup_{i=1}^I\surf_i$ and let $\sigma$ be the metric on $\surf$ that coincides with $\sigma_i$ on $\surf_i$. 

Then, photon spheres have the following local\footnote{which are derived without appealing to the asymptotic decay at infinity.} properties:
\begin{Prop}[Cederbaum  \cite{CederPhoto}]\label{prop}
Let $\spacetime$ be a geometrostatic spacetime and let $(\photo,p)\hookrightarrow\spacetime$ be a (generalized) photon sphere arising as the inner boundary of $\spacet$. Let $\mathfrak{H}:\photo\to\R$ denote the mean curvature of $(\photo,p)\hookrightarrow\spacetime$ and write 
\begin{align}
\left(\photo,p\right)=\left(\R\times\surf,-N^{2}dt^2+\sigma\right)= \bigcup_{i=1}^I\left(\R\times\surf_i,-N_i^2dt^2+\sigma_i\right),
\end{align}
where each $\photo_{i}=\R\times\surf_{i}$ is a connected component of $\photo$. Then the mean curvature $\mathfrak{H}_{i} := \mathfrak{H}\vert_{\photo_{i}}$
is constant on each connected component $\photo_{i}=\R\times\surf_i$ and the embedding $(\surf,\sigma)\hookrightarrow(\slice,g)$ is totally umbilic with constant mean curvature $H_i = \frac{2}{3}\mathfrak{H}_i$ on the component $\surf_i$. The scalar curvature of the component $(\surf_i,\sigma_i)$, $\my{\sigma_i}{\Scal}$, is a non-negative constant, namely 
\begin{align}\label{eq:scal}
\my{\sigma}{\Scal}_i=\frac{3}{2}H_{i}^{2}.
\end{align}
Moreover, the normal derivative of the lapse function $N$ in direction of the outward unit normal $\nu$ to $\surf$, $\nu(N)$, is also constant on every component $(\surf_i,\sigma_i)$, $\nu(N)_i := \nu(N)\vert_{\surf_{i}}$. For each $i\in\lbrace1,\dots,I\rbrace$, either $H_{i}=0$ and $\surf_{i}$ is a totally geodesic flat torus or $\surf_{i}$ is an intrinsically and extrinsically round CMC sphere for which the above constants are related via
\begin{align}\label{eq:prop}
N_iH_i&=2\nu(N)_i,\\\label{eq:prop2}
\left(r_iH_i\right)^2&=\frac{4}{3},
\end{align}
where
\begin{align}
r_i&:=\sqrt{\frac{\lvert\surf_i\rvert_{\sigma_i}}{4\pi}}
\end{align}
denotes the \emph{area radius} of $\surf_{i}$.
\end{Prop}

\newpage
\begin{lemma}\label{untrapped} 
Let the setting and notation be as in Proposition \ref{prop}.  Then  the mean curvature of $P^3$ is positive,
$\mathfrak{H} > 0$.  Equivalently, each component $\Sigma_i^2$ of $\Sigma^2$ has (constant) positive mean curvature, $H_i >0$.
\end{lemma}

\proof  {The proof  uses an argument similar to one used in the proof of Theorem~3.1 in \cite{GalMiao}, to which we refer the reader for details. The proof here makes use of a specific characterization of photon spheres.}  
%The crux of the argument is essentially the same as in the proof of Theorem~3.1 in \cite{GalMiao} to which we refer the reader for details.  

Suppose some component $\glue$ of $\surf$ has mean curvature $H_i \le 0$.     Let $S_R = \{r = R\}$ be a large `radial sphere' out on the asymptotically flat end.  Choose $R$ sufficiently large so that $S_R$ is mean convex, i.e. has positive mean curvature with respect to the outward normal.   

Now consider $M^3$ in the so-called Fermat metric $\hat g = N^{-2} g$.  The fact that $P^3$ is a photon sphere implies that $\surf$ is totally geodesic in $(M^3, \hat g)$.  (The reason for this is that null geodesics in a standard static spacetime project to geodesics in the canonical slice with respect to the Fermat metric; see e.g.\ ~\cite[Chapter 8]{Frankel} for details.)  By compactness, there exist points $p \in \glue$ and $q \in S_R$ such that $d_{\hat g}(p, q) = d_{\hat g}(\glue,S_R)$.  Let $\gamma:[0,\ell] \to M^3$ be a unit speed length minimizing geodesic from $p$ to $q$ in $(M^3, \hat g)$.  The existence of such a geodesic relies on the fact that $\surf$ is totally geodesic in $(M^3, \hat g)$:  Essentially, the other components of $\surf$  serve as `barriers' for the construction of such a minimizer.  Since $\gamma$ minimizes the $\hat g$-distance between $\glue$ and $S_R$, it meets these surfaces orthogonally.  Moreover,  there will be no $\hat g$-cut points to $\glue$ along $\gamma$, except possibly at the end point $q = \gamma(\ell)$.  Suppose 
also that $q$ is not a cut point (this is a technicality that can be easily handled).  Then for each $t \in [0, \ell]$, the set $W_t$ a $\hat g$-distance $t$ from $\glue$  will be a smooth surface near the point $\gamma(t) \in W_t$.  
By restricting to a neighborhood of $\gamma(t)$ we may assume each $W_t$, $t\in [0,\ell]$, is smooth.
 
From the minimizing property of $\gamma$, we see that $W_{\ell}$ lies to the inside of $S_R$, but touches $S_R$ at $q$.   For each $t \in [0,\ell]$, let $H(t)$ be the mean curvature of $W_{t}$  at $\gamma(t)$ in the {\it original metric} $g$.  By the maximum principle and mean convexity of $S_R$, $H(\ell) > 0$.  On the other hand, by the monotonicity formula \cite[Equation 3.24]{GalMiao}, the function  $t \to H(t)/N(\gamma(t))$ must be nonincreasing.  Since we are assuming that $H_i \le 0$,   this would imply that $H(\ell)  \le 0$, a contradiction.~\qed 

\begin{remark}
Lemma \ref{untrapped}  rules out the torus case in Proposition \ref{prop} and thus ensures that each photon sphere component is diffeomorphic to a cylinder over a sphere. Moreover, it ensures that not only $H_{i}>0$ but also $\nu(N)_{i}>0$ on each component $\surf_{i}$.   
\end{remark}

\begin{remark}
 In our definition of a standard static spacetime, we have assumed just one asymptotically flat end.  If, however, we had allowed several asymptotically flat ends in the definition, the argument used to prove Lemma \ref{untrapped} could  be used to prove in the context of Theorem~\ref{thm:main} that there can be only one end.
\end{remark}

\newpage
\section{Proof of the main theorem}\label{sec:proof}
This section is dedicated to the proof of the `static photon sphere uniqueness theorem':
\begin{theorem}\label{thm:main}
Let $\spacetime$ be a geometrostatic spacetime that possesses a (possibly disconnected) photon sphere $(\photo,p)\hookrightarrow\spacetime$, arising as the inner boundary of~$\spacet$. Let $m$ denote the ADM-mass of $\spacetime$ and let $\mathfrak{H}:\photo\to\R$ denote the mean curvature of $(\photo,p)\hookrightarrow\spacetime$.  Then $m=(\sqrt3\,\mathfrak{H})^{-1}$, with $\mathfrak{H}>0$, and $\spacetime$ is isometric to the region $\lbrace r\geq3m\rbrace$ exterior to the photon sphere $\lbrace r=3m\rbrace$ in the \schild spacetime of mass~$m$. In particular, $(\photo,p)$ is connected and a cylinder over a topological sphere.\end{theorem}

\begin{proof}[Proof of Theorem \ref{thm:main}]
We will work directly in the canonical slice $(\slice,g)$ throughout the proof. The main idea of our proof is as follows: In Step 1, we will define a new static asymptotically flat Riemannian manifold $(\widetilde{M}^3,\widetilde{g}\,)$ with (Killing) horizon boundary by gluing in some carefully chosen pieces of (spatial) \schild manifolds of appropriately chosen masses. More precisely, at each photon sphere base $\surf_{i}$, we will glue in a ``neck'' piece of a \schild manifold of mass $\mu_{i}>0$, namely the cylindrical piece between its photon sphere and its horizon. This creates a new horizon boundary corresponding to each $\surf_{i}$. The manifold $\widetilde{M}^{3}$ itself is smooth while the metric $\widetilde{g}$ is smooth away from the gluing surfaces, and, as will be shown, $C^{1,1}$ across them. Also, away from the gluing surfaces, $(\widetilde{M}^3,\widetilde{g}\,)$ is scalar~flat.

Then, in a very short Step 2, we double the glued manifold over its minimal boundary\footnote{In case the original manifold had additional Killing horizon boundary components, $M$ has additional boundary components with $N=0$ and thus $H=0$, $\free{h}=0$, $\my{\sigma\!}{R}\equiv\text{const.}$ The minimal boundary $\mathfrak{B}$ constructed here is thus of the same geometry and regularity as the spatial slices of the event horizon and the doubling can be carried out for both at once.} and assert that the resulting manifold---which will also be called $(\widetilde{M}^3,\widetilde{g}\,)$---is in fact smooth across $\mathfrak{B}$. The resulting manifold has two isometric ends and is geodesically complete.

In Step 3, along the lines of \cite{BMuA}, we will conformally modify $(\widetilde{M}^{3},\widetilde{g}\,)$ into another geodesically complete, asymptotically flat Riemannian manifold $(\widehat{M}^{3}=\widetilde{M}^{3}\cup\lbrace \infty\rbrace,\widehat{g}\,)$.  By our choice of conformal factor, $(\widehat{M}^{3},\widehat{g}\,)$ is smooth and scalar flat away from the gluing surfaces and the point $\infty$, and suitably regular across them. The new manifold $(\widehat{M}^{3},\widehat{g}\,)$ will have precisely one end of vanishing ADM-mass.

In Step 4, applying the rigidity statement of the positive mass theorem with suitably low regularity, we find that $(\widehat{M}^{3},\widehat{g})$ must be isometric to Euclidean space $(\R^{3},\delta)$. Thus, the original geometrostatic manifold $(\slice,g,N)$ was conformally flat and it follows as in \cite{BMuA} that it is indeed isometric to the exterior region $\lbrace r\geq3m\rbrace$ of the photon sphere in the (spatial) \schild manifold of mass $m=\frac{1}{\sqrt3\,\mathfrak{H}}>0$. This will complete our proof.

\subsection*{Step 1: Constructing a scalar flat, asymptotically flat manifold with minimal boundary}
First, we recall that every connected component $\surf_{i}$ must be a topological sphere by Proposition \ref{prop} and Lemma \ref{untrapped}.
For each $i\in\lbrace{1,\dots,I\rbrace}$, we define the \emph{mass} $m_i$ of $\surf_i$ by
\begin{align}\label{def:mi}
m_i&:=\frac{1}{4\pi}\int_{\surf_i}\nu(N)\,dA=\frac{\lvert\surf_i\rvert_{\sigma_i}}{4\pi}\nu(N)_{i}=r_{i}^{2}\,\nu(N)_{i},
\end{align}
where $dA$ denotes the area measure with respect to $\sigma$. 
This is motivated by \eqref{SMEvac3} and the relationship of ADM-mass and the asymptotic decay of the geometrostatic system $(\slice,g,N)$, see Section 4.2 in \cite{CDiss}. We find $m_{i}>0$ for all $i\in\lbrace1,\dots,I\rbrace$ by \eqref{eq:prop} and Lemma \ref{untrapped}.  Moreover, for each $i\in\lbrace{1,\dots,I\rbrace}$, we set 
\begin{align}\label{def:mui}
 \mu_i&:=\frac{r_i}{3},\quad I_i:=\left[2\mu_i,r_i=3\mu_{i}\right]\subset\R.
\end{align}

To each boundary component $\surf_i$ of $\slice$, we now  glue in a cylinder of the form $I_i\times\surf_i$. We do this such that the original photon sphere component $\surf_i\subset\slice$ corresponds to the level $\lbrace{r_i\rbrace}\times\surf_i$ of the cylinder $I_i\times\surf_i$ and will  continue to call this gluing surface $\surf_i$.
The resulting manifold $\widetilde{M}^{3}$ has inner boundary
\begin{align}\label{boundary}
\mathfrak{B}&:=\bigcup_{i=1}^{I}\,\lbrace2\mu_{i}\rbrace\times\surf_{i}.
\end{align}
{(The coordinate charts $(y^1,y^2, \psi)$ on $\widetilde{M}^{3}$ covering each $\surf_i$ are described below.)}

In the following, we construct an asymptotically flat, scalar flat Riemannian metric $\widetilde{g}$ on $\widetilde{M}^3$ which is smooth away from the gluing surfaces $\glue$, $i=1,\dots,I$, and smooth and geodesically complete up to the boundary $\mathfrak{B}$. On $\slice$, we keep the original metric $g$, so that $\widetilde{g}:=g$, there, while on $I_i\times\surf_i$, we set
\begin{align}\label{def:g}
 \widetilde{g}\vert_{I_i\times\surf_i}&:=\frac{1}{\varphi_i(r)^{2}}dr^2+\frac{r^{2}}{r_i^{2}}\sigma_i=\frac{1}{\varphi_i(r)^{2}}dr^2+r^{2}\,\Omega,\\\label{def:phi}
 \varphi_i(r)&:=\sqrt{1-\frac{2\mu_i}{r}},
\end{align}
where $r\in I_i$ denotes the coordinate along the cylinder $I_i\times\surf_i$ and we have used that $\sigma_i=r_{i}^{2}\,\Omega$  by Proposition \ref{prop}, with $\Omega$ the canonical metric on $\mathbb{S}^{2}$. In other words, for each $i\in\lbrace1,\dots,I\rbrace$, we have glued in the portion of spatial Schwarzschild of mass $\mu_i>0$ from the minimal surface to the photon sphere.  In a moment we will show that $\widetilde g$ is well defined and $C^{1,1}$ across $\surf$.

The metric $\widetilde{g}$ is naturally smooth away from the gluing surfaces $\glue$. The manifold $(\widetilde{M}^{3},\widetilde{g}\,)$ is geodesically complete up to the minimal boundary $\mathfrak{B}$ as $(\slice,g)$ was assumed to be geodesically complete up to the boundary in Definition \ref{def:AFSVS}. Moreover, it is scalar flat away from the gluing surfaces as both $(\slice,g)$ and (spatial) \schild are scalar flat.\\ 

\noindent
\emph{It remains to show that $(\widetilde{M}^{3},\widetilde{g}\,)$ is $C^{1,1}$ across all gluing surfaces.}  To show this, we introduce the function
\begin{align}\label{def:v}
\psi:\widetilde{M}^3\to\R: p\mapsto
   \begin{cases}
     N(p) & \text{if } p\in \slice,\\[0.25cm]
     \frac{3 m_i}{r_i}\,\varphi_i(r(p)) & \text{if } p\in I_i\times\surf_i
         \end{cases}
\end{align}
which we intend to use as a smooth collar function across the gluing surfaces $\glue$. {Let us now fix $i\in\lbrace{1,\dots,I\rbrace}$.}

{Let us first show that $\psi$ is indeed well-defined and can be used as a smooth coordinate in a neighborhood of the gluing surface $\glue$: First, by construction, $\psi$ is smooth away from $\glue$. Then, by our choice of constant factor\footnote{which equals $1$ in case $(\slice,g)$ already is a Schwarzschild manifold} $\frac{3m_i}{r_i}$ in front of $\varphi_i$, Proposition \ref{prop} and in particular Equation \eqref{eq:prop} imply that $\psi$ has the same constant value on each side of $\surf_i$, and hence is well-defined and continuous across $\glue$.}

{The outward unit normal vector to $\glue$ with respect to the \schild side is given by $\widetilde{\nu}=\varphi_{i}(r_{i})\,\partial_{r}$.  In $\slice$, the outward unit normal is given by $\widetilde{\nu}=\nu$. By our choice of $\mu_i$, one verifies, using Proposition \ref{prop}, \eqref{def:mi}, and \eqref{def:mui}, that the normal derivative of $\psi$ is the same {\it positive} constant on both sides of $\glue$. By using the integral curves of $\nabla\psi$ on $M^3$ (respectively, on $I_i\times\surf_i$), this allows us to use $\psi$ as a smooth coordinate function in a (collared) neighborhood of $\glue$ in $\widetilde{M}^{3}$. In particular, with respect to the coordinate charts $(y^A, \psi)$ introduced below, $\psi$ is automatically smooth across $\surf_i$. Moreover, as $\widetilde{\nu}(\psi)\vert_{\surf_i}$ coincides from both sides, the normal $\widetilde{\nu}$ is in fact continuous and thus by smoothness even $C^{0,1}$ across $\surf_i$. }

Let $(y^{A})$ be local coordinates on $\glue$ and flow them to a neighborhood of $\glue$ in $\widetilde{M}^{3}$ along the level set flow defined by $\psi$. To show that $\widetilde{g}$ is $C^{1,1}$ across $\glue$, it then suffices to show that the components $\widetilde{g}_{AB}$, $\widetilde{g}_{A\psi}$, and $\widetilde{g}_{\psi\psi}$ are $C^{1,1}$ with respect to the coordinates $(y^{A},\psi)$ across the $\psi$-level set $\glue$ for all $A,B=1,2$. 

 As $\widetilde{\nu}(\psi)\vert_{\surf_i}$ coincides from both sides as discussed above, the normal $\widetilde{\nu}$ is in fact continuous and thus even $C^{0,1}$ across $\glue$. 
Continuity of $\widetilde g$ in $(y^{A},\psi)$ and smoothness in tangential directions along $\glue$ is then immediate as $\partial_{\psi}=\frac{1}{\widetilde{\nu}(\psi)}\,\widetilde{\nu}$, and thus
\begin{align}\label{continuous}
\widetilde{g}_{AB}&=r_{i}^{2}\,\Omega_{AB},\quad \widetilde{g}_{A\psi}=0, \quad \widetilde{g}_{\psi\psi}=\frac{1}{(\widetilde{\nu}(\psi))^{2}}
\end{align}
on $\surf_{i}(r_i)^{\pm}$ for all $A,B=1,2$ and from both sides. Moreover, we find that
\begin{align*}
\partial_{\psi}\left(\widetilde{g}_{AB}\right)&=\frac{2}{\widetilde{\nu}(\psi)}\,\widetilde{h}_{AB}
\end{align*}
holds on $\glue$, where $\widetilde{h}_{AB}$ is the second fundamental form induced on $\glue$ by $\widetilde{g}$. 
Proposition \ref{prop} ensures umbilicity of every component of any photon sphere. Also, it asserts that the mean curvature of every component of any photon sphere is determined by its area radius via \eqref{eq:prop2}, up to a sign. Hence, using \eqref{continuous}, we find that $\widetilde{h}=\pm\frac{1}{2}H_{i}\sigma_{i}=\pm\frac{1}{2}H_{i}\,r_{i}^{2}\,\Omega$ holds on both sides of the photon sphere gluing boundary component $\glue$. We still need to determine this a priori free sign: From the side of $\slice$, we actually know $H_{i}>0$, where $H_{i}$ is computed with respect to the unit normal pointing towards the asymptotic end. On the \schild side, the mean curvature of the photon sphere with respect to the unit normal $\widetilde{\nu}$ pointing towards infinity and thus into $\slice$, is also positive. Thus, in both cases $\widetilde{h}_{AB}$ and thus also $\partial_{\psi}\left(\widetilde{g}_{AB}\right)$ coincide from both sides of $\glue$ for all $A,B=1,2$.

By construction, $\widetilde{g}_{A\psi}=0$ not only on $\glue$ but also in a neighborhood of $\glue$ inside $\widetilde{M}^{3}$ so that $\partial_{\psi}(\widetilde{g}_{A\psi})=0$ on both sides of $\glue$ for $A=1,2$. It remains to be shown that $\partial_{\psi}(\widetilde{g}_{\psi\psi})$ coincides from both sides. Then, because $\widetilde{g}$ is smooth on both sides up to the boundary (by assumption), we will have proved that $\widetilde{g}$ is $C^{1,1}$ everywhere.

A direct computation using the level set flow equations for $\psi$ shows that 
\begin{align*}
\partial_{\psi}(\widetilde{g}_{\psi\psi})&=-2\,(\widetilde{\nu}(\psi))^{2}\,\,\my{\widetilde{g}}{\nabla}^{2}\psi(\widetilde{\nu},\widetilde{\nu})
\end{align*}
from both sides of $\glue$, where $\my{\widetilde{g}}{\nabla}^{2}\psi$ denotes the Hessian of $\psi$ with respect to $\widetilde{g}$. We already know that $\widetilde{\nu}(\psi)$ is continuous across $\glue$. But from \eqref{surflap} we have
\begin{align}\label{Hessian}
\my{\widetilde{g}}{\nabla}^{2}\psi(\widetilde{\nu},\widetilde{\nu})=\cancelto{0}{\mylap{\widetilde{g}}\psi}-\cancelto{0}{\mylap{\widetilde{\sigma_{i}\!}}\psi}-\widetilde{H}_{i}\,\widetilde{\nu}(\psi)=-\widetilde{H}_{i}\,\widetilde{\nu}(\psi),
\end{align}
on both sides of $\glue$, where $\widetilde{H}_{i}$ denotes the mean curvature induced by $\widetilde{g}$ with respect to $\widetilde{\nu}$, and $\mylap{\widetilde{g}}$ and $\mylap{\widetilde{\sigma_{i}}}$ denote the $3$- and $2$-dimensional Laplacians induced by $\widetilde{g}$ and $\widetilde{\sigma_{i}}:=\widetilde{g}\vert_{T\surf_{i}\times T\surf_{i}}$, respectively. (Here, we have used that $\psi$ is constant along $\glue$ and that $\psi$ is $\widetilde{g}$-harmonic on both sides of the photon sphere gluing boundary.) To conclude, we recall that $\widetilde{H}_{i}$ and $\widetilde{\nu}(\psi)$ are continuous across $\glue$ so that $\my{\widetilde{g}}{\nabla}^{2}\psi(\widetilde{\nu},\widetilde{\nu})$ and thus $\partial_{\psi}(\widetilde{g}_{\psi\psi})$ are continuous across $\glue$. Thus, $\widetilde{g}$ is $C^{1,1}$ across $\glue$. As $i\in\lbrace1,\dots,I\rbrace$ was arbitrary, $\widetilde{g}$ is indeed $C^{1,1}$ on all of $\widetilde{M}^{3}$. 

\subsection*{Step 2: Doubling}
Now, we rename $\widetilde{M}^{3}$ to $\widetilde{M}^{+}$, reflect $\widetilde{M}^{+}$ through $\mathfrak{B}$ to obtain $\widetilde{M}^{-}$, and  glue the two copies to each other along their shared minimal boundary~$\mathfrak{B}$. We thus obtain a new smooth manifold which we will call $\widetilde{M}^3$ by a slight abuse of notation. We define a metric on $\widetilde{M}^{3}$ by equipping both $\widetilde{M}^{\pm}$ with the metric $\widetilde{g}$ constructed in Step 1 and extend the function $\psi^{+}:=\psi$ constructed in Step 1 on $\widetilde{M}^{+}$ across $\mathfrak{B}$ by choosing $\psi^{-}:=-\psi^{+}$ on $\widetilde{M}^{-}$. Combined, we will again abuse notation and call the extended function $\psi:=\pm\psi^{+}$ on $\widetilde{M}^{\pm}$.\\

 \noindent \emph{Why $(\widetilde{M}^{3},\widetilde{g}\,)$ is smooth across $\mathfrak{B}$ and $\psi$ can be used as a smooth collar coordinate function near $\mathfrak{B}$.}  In contrast to \cite{BMuA}, this is in fact immediate because on each component $\lbrace2\mu_{i}\rbrace\times\surf_{i}$ of $\mathfrak{B}$, we are just gluing two \schild necks of the same mass $\mu_{i}$ to each other across their minimal surface boundaries (up to a constant factor in their lapse functions $\varphi_{i}$, see  \eqref{def:phi}). Indeed, $\psi$ is smooth across the horizon boundary $\mathfrak{B}$ as can be seen in isotropic coordinates. Smoothness of the metric $\widetilde{g}$ across $\mathfrak{B}$ then follows directly.\\

By construction, the doubled manifold $(\widetilde{M}^{3},\widetilde{g}\,)$ has two isometric asymptotically flat ends of mass $m$. It is geodesically complete as $(\slice,g)$ was assumed to be geodesically complete up to the boundary in Definition \ref{def:AFSVS}. Finally, we observe that $\psi$ is $\widetilde{g}$-harmonic away from $\glue$ by construction.

\subsection*{Step 3: Conformal transformation to a scalar flat, geodesically complete manifold with vanishing ADM-mass}
As in Bunting and Masood-ul-Alam \cite{BMuA}, we want to use $u:=\frac{1+\psi}{2}$ as a conformal factor on $\widetilde{M}^{3}$. However, in our situation it is not a priori evident that $u>0$, or in other words that $\psi>-1$ on $\widetilde{M}^{3}$. \\

\paragraph*{\emph{Why $\psi>-1$ holds on $\widetilde{M}^{3}$}.}
By the reflection symmetric definition of $\psi$ in Step 2, it suffices to show that $0\leq\psi<1$ in $\widetilde{M}^{+}$ or in other words $0\leq\psi<1$ on $M^{+}$ and on each neck $(I_{i}\times\surf_{i})^{+}$, $i\in\lbrace1,\dots,I\rbrace$. On $M^{+}=\slice$, $\psi=N>0$ and we know that $N\to1$ as $r\to\infty$. By the maximum principle for elliptic PDEs (see e.\,g.\ \cite{GT}), it thus suffices to show that $N_{i}<1$ on each boundary component $\surf_{i}$. Assume this was false. Let $i_{0}$ be such that $N_{i_{0}}=\max\lbrace N_{i}\;\vert\; i=1,\dots,I\rbrace\geq1$. Then, again by the maximum principle, $N$ attains its global maximum at $p_{0}\in\surf_{i_{0}}$ and thus $\nu(N)\vert_{p_{0}}\leq0$. However, we already know that $\nu(N)\vert_{p_{0}}=\nu(N)_{i_{0}}>0$, a contradiction. Thus $0<\psi=N<1$ on $\slice$.

On each neck $(I_{i}\times\surf_{i})^{+}=(I_{i}\times\surf_{i})$, $\psi=\frac{3 m_i}{r_i}\,\varphi_i$, where $\varphi_{i}$ is the Schwarzschild lapse function given by \eqref{def:phi}. It is easy to compute that $\varphi_{i}(I_{i})=\left[0,1/\sqrt{3}\,\right]$, so that $\psi(I_{i}\times\surf_{i})=\left[0,\sqrt{3}m_{i}/r_{i}=N_{i}\right]$, where we have used \eqref{eq:prop}, \eqref{eq:prop2}, and \eqref{def:mi}. But we have just argued why $N_{i}<1$ for all $i\in\lbrace1,\dots,I\rbrace$ and can thus conclude that $\vert\psi\vert<1$ everywhere on $\widetilde{M}$. \\

The above considerations allow us to define the conformal factor
\begin{align}\label{def:u}
 u:\widetilde{M}^3\to\R,\,u:=\frac{1+\psi}{2}>0
\end{align}
and a conformally transformed Riemannian metric
\begin{align}\label{def:ghat}
\widehat{g}&:=u^4\,\widetilde{g}
\end{align}
on $\widetilde{M}^3$. Away from the gluing surfaces, $u$ is smooth and $\widetilde{g}$-harmonic because $\psi$ is smooth and $\widetilde{g}$-harmonic there.  Since, in addition, $\widetilde{g}$ is smooth and scalar flat away from the gluing surfaces, the same holds true for $\widehat{g}$ as
\begin{align}
\my{\widehat{g}\!}{R}&=u^{-5}\left(\cancelto{0}{\my{\widetilde{g}\!}{R}}\,u+8\,\cancelto{0}{\mylap{\widetilde{g}}u}\,\,\,\,\,\right)=0.
\end{align}

Furthermore, $u$ and $\widehat{g}$ are $C^{1,1}$ across all gluing boundaries $\glue$, $i\in\lbrace1,\dots,I\rbrace$, because $\psi$ and $\widetilde{g}$ 
are $C^{1,1}$ there (by product rule and \eqref{def:u}, \eqref{def:ghat}).%\vfill\newpage

Moreover, precisely as in \cite{BMuA}, $(M^+,\widehat{g}\,)$ is asymptotically flat with zero ADM-mass. 

Heuristically, assuming that $(\slice,g)$ is asymptotically Schwarzschildean, one easily sees that
\begin{align*}
\widehat{g}_{ij}&=u^4\, \widetilde{g}_{ij}=\left(\frac{1+\psi}{2}\right)^4 \widetilde{g}_{ij}=\left(\frac{1+N}{2}\right)^4 g_{ij}\\
&=\left(1-\frac{m}{2r}+\mathcal{O}_{k}\!\left(\frac{1}{r^2}\right)\right)^4 \left(\left(1+\frac{m}{2r}\right)^4\delta_{ij}+\mathcal{O}_{k}\!\left(\frac{1}{r^2}\right)\right)=\delta_{ij}+\mathcal{O}_{k}\!\left(\frac{1}{r^2}\right)
\end{align*}
as $r\to\infty$ in the original asymptotically flat coordinates of $(\slice,g)$. On the other hand, again precisely as in \cite{BMuA}, $(M^-,\widehat{g}\,)$ can be compactified by adding in a point $\infty$ at infinity because
\begin{align*}
\widehat{g}_{ij}&=u^4\, \widetilde{g}_{ij}=\left(\frac{1+\psi}{2}\right)^4 \widetilde{g}_{ij}=\left(\frac{1-N}{2}\right)^4 g_{ij}\\
&=\left(\frac{m}{2r}+\mathcal{O}_{k}\!\left(\frac{1}{r^2}\right)\right)^4 \left(\left(1+\frac{m}{2r}\right)^4\delta_{ij}+\mathcal{O}_{k}\!\left(\frac{1}{r^2}\right)\right)=\frac{m^4}{16r^4}\delta_{ij}+\mathcal{O}_{k}\!\left(\frac{1}{r^5}\right)
\end{align*}
as $r\to\infty$ in the original asymptotically flat coordinates of $(\slice,g)$. {Heuristically, again assuming that $(M^{3},g)$ is asymptotically Schwarzschildean, one can perform an inversion in the sphere via $R:=r^{-1}$ and $X^{i}:=\frac{x^{i}}{r^2}$ to find that }
\begin{align*}
\widehat{g}\left(\partial_{X^i},\partial_{X^j}\right)=\left(\frac{m}{2}\right)^4\delta_{ij}+\mathcal{O}_{k}(R) 
\end{align*}
as $R\to0$. This {heuristic argument illustrates why it is allowed to glue in a point $\infty$ at $R=0$ with $(\widehat{M}^{3}:=\widetilde{M}^3\cup\lbrace \infty\rbrace,\widehat{g}\,)$ sufficiently regular. By construction, $(\widehat{M}^{3},\widehat{g}\,)$ is geodesically complete.}

Summarizing, we now have constructed a geodesically complete, scalar flat Riemannian manifold $(\widehat{M}^{3},\widehat{g}\,)$ with one asymptotically flat end of vanishing ADM mass that is smooth away from some hypersurfaces and one point. {At the point $\infty$ as well as at all gluing surfaces, the regularity is precisely that encountered by \cite{BMuA}.}

\subsection*{Step 4: Applying the Positive Mass Theorem.}
In Steps 1-3, we have constructed the geodesically complete, scalar flat Riemannian manifold $(\widehat{M}^{3},\widehat{g}\,)$ with one asymptotically flat end of vanishing ADM mass in a manner similar to what is done in  \cite{BMuA}.  Moreover, as noted above, the regularity achieved is the same as that encountered in \cite{BMuA}.  As such,  their analysis fully applies and asserts that the (weak regularity) Positive Mass Theorem proved by Bartnik \cite{Bartnik} applies.  (See also Lee and LeFloch  \cite{LeeLefloch} where even weaker regularity is allowed.)

The rigidity statement of this (weak regularity) Positive Mass Theorem implies that $(\widehat{M}^{3},\widehat{g}\,)$ must be isometric to Euclidean space. This immediately shows that the photon sphere $\photo$ was connected and diffeomorphic to a cylinder over a sphere for topological reasons. Moreover, it allows us to deduce that $(\slice,g)$ must be conformally flat. It is well-known\footnote{and can be verified by a straightforward computation, see e.\,g.\ \cite{Wald} or \cite{BMuA}.} that the only conformally flat, maximally extended solution of the static vacuum equations \eqref{SMEvac1}, \eqref{SMEvac2} is the \schild solution \eqref{schwarzmetric}.

In particular, the lapse function $N$ is given by $N=\sqrt{1-2m/r}$ with $r$ the area radius along the level sets of $N$, and $m$ is the ADM-mass of $(\slice,g)$ as before. Equation \eqref{SMEvac3} then shows that $\mu_1=m$ from which we find that $r_1=3m$ and $\mathfrak{H}=\mathfrak{H}_1=\sqrt{3}/r_1=1/(\sqrt{3}m)$, using the algebraic relations \eqref{eq:prop}, \eqref{eq:prop2}. This proves the claim of Theorem \eqref{thm:main}.
\end{proof}

\section{The static $n$-body problem for very compact bodies and black holes}\label{sec:nbody}
The following theorem addresses the so-called ``static $n$-body problem'' in General Relativity, namely the question whether multiple suitably ``separated'' bodies and black holes can be in static equilibrium. So far, only limited progress has been made towards settling this question. Bunting and Masood-ul-Alam's result \cite{BMuA} on static black hole uniqueness can be re-interpreted as saying that there are no $n>1$ black holes in static equilibrium. M\"uller zum Hagen \cite{MzH} and Beig and Schoen \cite{BeigSchoen} showed that static $n$-body configurations cannot exist with $n>1$ in axisymmetry and in the presence of an (infinitesimal) mirror symmetry, respectively. Both exclude black holes from their considerations. Andersson and Schmidt \cite{AS} constructed static equilibrium configurations of multiple elastic bodies, albeit not very ``separated'' ones.

Our approach can handle both bodies and black holes simultaneously and makes no symmetry assumptions. However, we can only treat ``very compact'' bodies, namely bodies that are each so compact that they give rise to a photon sphere behind which they reside:

\begin{theorem}[No static configuration of $n$``very compact'' bodies and black holes]\label{thm:nbody}
There are no static equilibrium configurations of $n>1$ bodies and black holes in which each body is surrounded by its own photon sphere.
\end{theorem}

To be specific, the term \emph{static equilibrium} is interpreted here as referring to a geometrostatic system $(\overline{M}^{3},\overline{g},\overline{N})$ as defined in Definition \ref{def:AFSVS}, geodesically complete up to (possibly) an inner boundary consisting of one or multiple \emph{black holes}, defined as sections of a Killing horizon (or in other words consisting of totally geodesic topological spheres satisfying $N=0$). Furthermore, a \emph{body} is meant to be a bounded domain $\Omega\subset\overline{M}^{3}$ where the static equations hold with a right hand side coming from the energy momentum tensor of a matter model satisfying the dominant energy condition $\my{\overline{g}}{\Scal}\geq0$. We consider a body $\Omega$ to be \emph{very compact} if it creates a photon sphere $\surf$ that, without loss of generality, arises as its boundary, $\surf=\partial\Omega$. Naturally, all bodies are implicitly assumed to be disjoint. Outside all bodies, the system is assumed to satisfy the static vacuum 
equations \eqref{SMEvac1}, \eqref{SMEvac3}.

Theorem \ref{thm:main} then applies to the spacetime $(\mathfrak{L}^{4}=\R\times\slice,\mathfrak{g}=-\overline{N}^{2}dt^{2}+\overline{g}\vert_{\slice})$, where $\slice:=\overline{M}^{3}\setminus\cup_{i=1}^{I}\Omega_{i}$, and $\emptyset\neq\Omega_{i}\subset\overline{M}^{3}$, $i\in\lbrace1,\dots,I\rbrace$, are all the bodies in the system. We appeal to Remark \ref{rem:multiple} if black holes are present in the configuration.\qed

\smallskip
Indeed, while photon spheres might be most well known from the vacuum \schild spacetime \eqref{schwarzmetric}, many astrophysical objects are believed to be surrounded by photon spheres, see e.\,g.\ \cite{yazadjiev} and references cited therein. To give an example of a static matter fill-in of the \schild region exterior to the photon sphere, we refer to the well-known \schild interior solution constructed as a perfect fluid, see for example \cite[p.~130ff.]{Frankel}. In this example, the interior fluid ball can be glued into any exterior region of a \schild spacetime of mass $m>0$ where the radius of the centered round gluing sphere has area radius $r$ satisfying the Buchdahl condition
\begin{align}
\frac{2m}{r}<\frac{8}{9}.
\end{align} 
This condition is easily satisfied by the photon sphere which has area radius $r=3m$.

\vfill\newpage
\bibliographystyle{amsplain}
\bibliography{photon-sphere}
\vfill
\end{document}